\documentclass[a4paper, english, 10pt]{amsart}

\usepackage{amsfonts}
\usepackage{caption}
\usepackage{forest}
\usepackage{amsmath}
\usepackage{amssymb}
\usepackage{tikz}
\usepackage[cp850]{inputenc}
\usepackage[bookmarksnumbered,plainpages]{hyperref}
\usepackage{color}
\usepackage{mathtools}
\usepackage{MnSymbol}
\usepackage{subfigure}
\usepackage[all]{xy}
\allowdisplaybreaks
\newtheorem{theorem}{Theorem}[section]

\newtheorem{proposition}[theorem]{Proposition}
\newtheorem{corollary}[theorem]{Corollary}

\newtheorem{lemma}[theorem]{Lemma}

\newtheorem{definition}[theorem]{Definition}
\newtheorem{example}[theorem]{Example}

\def\aut#1{\mathrm{Aut}(#1)}

\def\lmlt{\mathrm{RMlt}}

\def\setof#1#2{\{#1\, : \,#2\}}

\newcommand*\xbar[1]{%
   \hbox{%
     \vbox{%
       \hrule height 0.5pt 
       \kern0.5ex
       \hbox{%
         \kern-0.1em
         \ensuremath{#1}%
         \kern-0.1em
       }%
     }%
   }%
}

\makeindex 
\begin{document}

\markboth{Authors' Names}{Instructions for Typesetting Manuscripts using \LaTeX}


\title{Oriented singquandles and related algebraic structures}

\author{M. Bonatto}

\address{Department of Mathematics and Computer Science, University of Ferrara, Via Macchiavelli 30, 44121 Ferrara, Italy\\ marco.bonatto.87@gmail.com}
%
%



\begin{abstract}
In this paper we consider the algebraic structures related to invariants of topological structures introduced respectively in \cite{Stuq} and \cite{Bondles}. Our main results is to show how all these structures are closely related to each other using the language of binary operations. 
\end{abstract}
\maketitle

\section*{Introduction}

Algebraic invariants are relevant tools in topology, in particular towards the study of knots and virtual knots. The study of invariants based on binary algebraic structures have been started by the seminal work of Joyce \cite{J} and Matveev \cite{Matveev} that first introduced quandles as binary algebraic structures encoding the classical Reidemeister moves. The same approach has been adapted to virtual knots and the related Reidemeister moves and so {\it oriented singquandles} have been defined \cite{OrientedSing}. These structures can be studied in a purely algebraic way and an axiomatization of them (alternative to the original one) can be found in \cite{NostroSing}.

Similarly, some algebraic structures have also been defined in the context of topological study of proteins \cite{Bondles} and RNA folding \cite{Stuq} in order to produce invariants able to distinguish topological objects.

In the present paper we offer new axiomatizations of all such structures, namely {\it stuquandles} and {\it oriented bondles} and we show that all of them can be reduced to the concept of oriented singquandle. Indeed stuquandles are just given by a pair of oriented singquandles structure over the same underlying set (see Corollary \ref{corollary stuq}) and oriented bondles are stuquandles with an additional compatibility condition between the quandle and one of the two oriented singquandles structures (see Theorem \ref{theorem for bondles}).

Our approach is purely algebraic so we do not provide any detail on the underlying topology and we refer the reader to the above mentioned papers for all the missing information. Nonetheless, we wonder if a topological argument can show the same relation between the algebraic structures we deal with in the paper.

The paper is organized as follows: in the first section we recall the basics about binary algebraic structures and right quasigroups and quandles in particular. The second one is about Oriented singquandles, and we resume the main results of \cite{NostroSing}. In the next two sections we deal with stuquandles and oriented bondles and we show that such structured can be build from oriented singquandles. In the last section we show that the family of 2-reductive quandles provide examples of oriented bondles. 


\section{Algebraic structures}

\subsection{Binary algebraic structures}

A binary algebraic structure is a set $Q$ endowed with a set of binary operations. Let $(Q,\cdot)$ be a binary algebraic structure, we can define the left and right multiplications mappings with respect to $\cdot$ as 
\begin{align*}
L^{[\cdot]}_x:y\mapsto x\cdot y,\quad  R_x^{[\cdot]}:y\longrightarrow y\cdot x
\end{align*}
for every $x\in Q$ (we often denote such mappings just as $L_x$ and $R_x$). We can define the opposite of $\cdot$ as $x\cdot^{op} y =y\cdot x$ for every $x,y\in Q$.

Binary operations and binary functions are essentially the same thing. Indeed, given a binary algebraic structure $(Q,\cdot)$ we can define the map $M[\cdot](x,y)=x\cdot y$. On the other hand, given a binary map $f:Q\times Q\longrightarrow Q$ we can define the binary operation $x\cdot_f y=f(x,y)$ on $Q$.

A bijective map $f$ on $Q$ is an automorphism of $(Q,\cdot)$ if $f(x\cdot y)=f(x)\cdot f(y)$ for every $x,y\in Q$. We denote the group of automorphism of $(Q,\cdot)$ as $\aut{Q,\cdot}$.

A {\it right quasigroup} is a binary algebraic structure $(Q,\cdot,/)$ such that 
$$(x\cdot y)/y\approx x \approx (x/y)\cdot y$$
hold. The right multiplication mappings of $(Q,\cdot)$ are bijective and in particular $x/y=R_y^{-1}(x)$ for every $x,y\in Q$. Given $n\in \mathbb{Z}$ we define $(Q,\cdot^n,/^n)$ where 
$$x \cdot^n y=R_y^n(x),\quad x/^n y=R_y^{-n}(x).$$
Note that $(Q,\cdot^n,/^n)$ is also a right quasigroup.

We say that $(Q,\cdot,/)$ is:
\begin{itemize}
\item[(i)] {\it idempotent} if $x\cdot x \approx x$;
\item[(ii)]  {\it projection}, if $x\cdot y\approx x$;
\item[(iii)] {\it involutory} if $(x\cdot y)\cdot y\approx x$ (namely $\cdot=/$);
\item[(iv)] a {\it rack} if $ (x\cdot y)\cdot z\approx (x\cdot z)\cdot(y\cdot z)$.
\end{itemize}

Idempotent racks are called {\it quandles}. In the paper we deal with sets with several binary operations as $(Q,\cdot, *,/)$ and usually $(Q,*,/)$ is a quandle. We denote by $\rho_x$ the right multiplication with respect to $*$ for every $x\in Q$.

\subsection{Oriented singquandles}

\label{oriented}

In \cite{OrientedSing} an algebraic structure related to invariants of singular knots was introduced. An {\it oriented singquandle} is a tuple $(Q,*,/,R_1,R_2)$ where $(Q,*,/)$ is a (right) quandle and $R_1, R_2:Q\times Q\longrightarrow Q$ such that the following axioms hold:
\begin{align}
R_1(x,y)*z &=R_1(x*z,y*z) \label{O1}	\tag{OS1}\\
R_2(x,y)*z &=R_2(x*z,y*z)\label{O2}\tag{OS2}\\
(y*x)*z &=(y*R_1(x,z))*R_2(x,z) \label{O3}\tag{OS3}\\
R_1(x,y)*R_2(x,y) &=R_2(y,x*y) \label{O5}\tag{OS4}\\
R_2(x,y) &=R_1(y,x*y)\label{O4}\tag{OS5}
\end{align}

Setting $y\cdot x=R_1(x,y)$, and taking \eqref{O4} as the definition of $R_2$ by $\cdot$ and $*$ as $R_2(x,y)=(x*y)\cdot y$,  we can rewrite the definition as follows. 

\begin{definition}\label{Oriented def2}
An {\it oriented singquandle} is a binary algebraic structure $(Q,\cdot,*,/)$ where $(Q,*,/)$ is a (right) quandle and
such that the following identities hold:
\begin{align}
(y\cdot x)*z &\approx (y*z)\cdot (x*z) \label{O21}\tag{OS1'}	\\
(x\cdot (y\cdot x))*z &\approx(x*z)\cdot((y*z)\cdot (x*z)) \label{O22}\tag{OS2'}\\
(y*x)*z &\approx(y*(z\cdot x))*((x*z)\cdot z)) \label{O23}\tag{OS3'}\\
(y\cdot x)*((x*y)\cdot y) &\approx(y*(x*y))\cdot (x*y) \label{O25}\tag{OS4'}
%
\end{align}
\end{definition}

Clearly, if $(Q,*,/,R_1,R_2)$ is an oriented singquandle in the sense of \cite[]{OrientedSing}, then $(Q,(\cdot_{R_1})^{op},*,/)$ is an oriented sinquandle in the sense of Definition \ref{Oriented def2}. On the other hand, if $(Q,\cdot,*,/)$ satisfies the axioms in Definition \ref{Oriented def2}, then $(Q,*,/,M[{\cdot^{op}}],T[\cdot])$ where $T[\cdot](x,y)=(x\cdot y)*y$  is an oriented singquandle in the sense of \cite[]{}. Such assignements provide a one-to-one correspondence between the two type of structures. Thus, we stick to \ref{Oriented def2} as the definition of oriented singquandles.

We already showed in \cite{NostroSing} that the axioms above can be simplified as follows.

\begin{proposition}\label{Car Oriented}\cite[Proposition 3.2]{NostroSing}
Let $(Q,\cdot,*,/)$ be a binary algebraic structure. The following are equivalent:
\begin{itemize}
\item[(i)] $(Q,\cdot,*,/)$ is an oriented singquandle.
\item[(ii)] $(Q,*,/)$ is a quandle and the following identities hold:
\begin{align}
(x\cdot y)*z&\approx(x*z)\cdot (y*z),\label{Oriented1}\\
(z*y)*x &\approx(z*(x\cdot y))*((y*x)\cdot x)\label{Oriented2}.
\end{align}

\end{itemize} 
\end{proposition}

Note that if $(Q,*,/)$ is a projection quandle, $(Q,\cdot,*,/)$ is an oriented singquandle for every binary operation $\cdot$. On the other hand if $(Q,\cdot)$ is a projection quandle, then $(Q,\cdot,*,/)$ is an oriented singquandle for every quandle $(Q,*,/)$.

\begin{example}
Let $(Q,*,/)$ be a quandle. Then $(Q,*,*,/)$ is an oriented singquandle if and only if 
\begin{align*}
\rho_x \rho_y=\rho_{(y*x)*x}\rho_{x*y}&=\rho_x^2 \rho_y \rho_x^{-2} \rho_y \rho_x \rho_y^{-1}\\
&=\rho_x \rho_{y*x}\rho_{y/x}\rho_y^{-1}
\end{align*}
i.e. $\rho_y^2=\rho_{y*x}\rho_{y/x}$. Note that if $(Q,*,*,/)$ is an oriented singquandle if and only if $(Q,/,/,*)$ is an oriented sinquandle. For instance if $(Q,*,/)$ is involutory then $(Q,*,*,/)$ is an oriented singquandle.
\end{example}



\subsection{Stuquandles}

In \cite{Stuq} a new algebraic structure related to stuck knots have been defined. Let us recall the definition: let $(Q, *,/,R_1,R_2)$ be an oriented singquandle and let $R_3,R_4:Q\times Q\longrightarrow Q$. If the following axioms hold for all $x, y, z \in Q$
\begin{align}
R_3(y, x) * R_4(y, x) &= R_4(x * y, y),\label{eq1_0}\tag{ST1}\\
R_4(y, x) &= R_3(x * y, y),\label{eq2}\tag{ST2}\\
R_3(y * x, z) &= R_3(y, z/x) * x,\label{eq3}\tag{ST3}\\
R_4(y, z/x) &= R_4(y* x, z)/x,\label{eq4}\tag{ST4}\\
(x * R_4(y, z))/y &= (x/R_3(y, z))* z.\label{eq5}\tag{ST5}
\end{align}
we say that $(Q,*,R_1,R_2,R_3,R_4)$ is a stuquandle.


Similarly to what we did for Oriented singquandles we are going to translate the definition in terms of binary operations.

If we denote by $R_3(x,y)=x\circ y$, then \eqref{eq2} shows that $R_4$ is defined in terms of $*$ and $\circ$ as $R_4(x,y)=(y*x)\circ x$ for every $x,y\in Q$. Moreover, \eqref{eq3} is just saying that $\rho_x\in \aut{Q,\circ}$ for every $x\in Q$. The identity \eqref{eq4} follows since $\rho_x \in \aut{Q,*}$ and $R_4$ is defined in terms of $*$ and $\circ$. Hence we can rewrite the definition above as follows.

\begin{definition}
A stuquandle is a binary algebraic structure $(Q,\cdot,\circ,*,/)$ such that $(Q,\cdot,*,/)$ is an oriented sinquandles and the following identities hold:
\begin{align}
(y\circ x) * ((x*y)\circ y) &\approx (y*(x*y))\circ (x * y)\label{eq1}\tag{ST1'},\\
(x \circ y)* z &\approx (x*z)\circ (y*z),\label{eq3_2}\tag{ST3'}\\
(x * ((z*y)\circ y)))/y &\approx (x/(y\circ z))* z. \label{eq4_2}\tag{ST5'}
\end{align}
\end{definition}
%

The next theorem shows that equations \eqref{eq1}, \eqref{eq3_2} and \eqref{eq4_2} provide an equivalent axiomatizations of oriented singquandles.
\begin{theorem}
Let $(Q,*,/)$ be a quandle. The following are equivalent:
\begin{itemize}
\item[(i)] the binary algebraic structure $(Q,\circ,*,/)$ satisfies \eqref{eq1}, \eqref{eq3_2} and \eqref{eq4_2}.
\item[(ii)] $(Q,\circ,*,/)$ is an oriented singquandle (in the sense of Definition \ref{Oriented def2}).
\end{itemize}
\end{theorem}

\begin{proof}
Note that \eqref{eq4_2} is equivalent to have $\rho_y^{-1}	\rho_{(z*y)\circ y}=\rho_z\rho_{y\circ z}^{-1}$, i.e. $\rho_y\rho_z=\rho_{(z*y)\circ y}\rho_{y\circ z}$. Namely \eqref{eq4_2} can be rewritten as
\begin{align}
(x*z)*y =(x*(y\circ z))*((z*y)\circ y).\label{ST5''}\tag{ST5''}
\end{align}
Note that \eqref{eq1} is the very same identity as \eqref{O25}, and so according to Proposition \eqref{Car Oriented} it follows by  \eqref{eq3_2} and \eqref{ST5''}.
%
\end{proof}

\begin{corollary}\label{corollary stuq}
Let $(Q,\cdot,\circ,*,/)$ be a binary algebraic structure. The following are equivalent:
\begin{itemize}
\item[(i)] $(Q,\cdot,\circ,*,/)$ is a stuquandle.
\item[(ii)] $(Q,\cdot,*,/)$ and $(Q,\circ,*,/)$ are oriented singquandles. 
\end{itemize} 
\end{corollary}

\begin{example}
Let $(Q,\cdot,*,/)$ be an oriented singquandle. Then $(Q,\cdot, \cdot, *,/)$ is a stuquandle. 
\end{example}

\subsection{Oriented Bondles}


Let us introduce the algebraic structure defined in \cite{Bondles} in connection with the topological study of the structure of proteins. Let $(Q, *,/,R_1,R_2)$ be an oriented singquandle and let $R_3:Q\times Q\longrightarrow Q$. If the following axioms hold for all $x, y, z \in Q$
\begin{align}
R_3(y*z, x *z)& = R_3(y, x) * z\label{ax1}\tag{OB1}\\
R_3(x/ z, y / z) &= R_3(x , y) / z\label{ax2}\tag{OB2}\\
(z / R_3(x, y)) * x &= (z /y) * R_3(y, x)\label{ax3}\tag{OB3}\\
R_3(x, y) /y &= R_3(x /R_3(y, x), y).\label{ax4}\tag{OB4}
\end{align}
we say that $(Q,*,R_1,R_2,R_3)$ is an oriented bondle.

Let us switch to binary operations by setting $R_3(x,y)=y\bullet x$. The identities \eqref{ax1} and \eqref{ax2} are both equivalent to have that $\rho_x\in \aut{Q,\bullet}$ for every $x\in Q$, thus we can keep just \eqref{ax1}. 

\begin{definition}
An oriented bondle is a binary algebraic structure $(Q,\cdot,\bullet,*,/)$ such that $(Q,\cdot,*,/)$ is an oriented singquandle and the following identities hold:
\begin{align}
(x\bullet y)*z&\approx(x*z)\bullet (y*z),\tag{OB1'}\\
(z / (y\bullet x)) * x &\approx(z /y) * (x\bullet  y),\label{ax32b}\tag{OB3'}\\
(y\bullet x) /y &\approx y\bullet((x /(x\bullet y)) .\label{ax42b}\tag{OB4'}
\end{align}
\end{definition}

The following theorem shows that also oriented bondles are constructed by a pair of oriented singquandles with an additional compatibility condition.

\begin{theorem}\label{theorem for bondles}
Let $(Q,\circ,\bullet,*)$ be a binary algebraic structure. The following are equivalent:
\begin{itemize}
\item[(i)] $(Q,\circ,\bullet,*,/)$ is an oriented bondle.
\item[(ii)] $(Q,\circ,*,/)$ and $(Q,\bullet,*,/)$ are oriented singquandles and 
\begin{align}\label{eq_bondles}
(y* x)\bullet x \approx (y*(x \bullet y))\bullet x 
\end{align} holds.
\end{itemize}
\end{theorem}
\begin{proof}

The identity \eqref{ax32b} is equivalent to 
\begin{align*}
\rho_{x\bullet y} \rho_y^{-1} &=\rho_x \rho_{y \bullet x}^{-1}\\
&=\rho_x 	\rho_{y\bullet x}^{-1} \rho_x^{-1}\rho_x=\rho_{(y\bullet x)*x}^{-1}\rho_x
\end{align*}
i.e.
\begin{align*}
\rho_x \rho_y(z)&=(z*y)*x\\
\rho_{(y\bullet x)*x}\rho_{x\bullet y}(z)&=(z*(x\bullet y))*((y\bullet x)*x)\\
&=(z*(x\bullet y))*((y* x)\bullet (x*x))\\&=(z*(x\bullet y))*((y* x)\bullet x)
\end{align*}
Therefore $(Q,\bullet,*,/)$ is an oriented singquandle. 
%
%
%

Let us consider \eqref{ax42b} and write is as
\begin{align*}
y\bullet x&=(y\bullet (x/(x\bullet y))*y=y\bullet ((x*y)/((x*y)\bullet y))\\&=((y*((x*y)\bullet y))\bullet (x*y))/((x*y)\bullet y)=\rho_{(x*y)\bullet y}^{-1}((y*((x*y)\bullet y))\bullet (x*y)).
\end{align*}
Since $\rho_{(x*y)\bullet y}^{-1}=(\rho_y \rho_x \rho_{y\bullet x}^{-1})^{-1}=\rho_{y\bullet x} \rho_x^{-1} \rho_y^{-1} $ we have that 
\begin{align*}
 \rho_x \rho_{y\bullet x}^{-1} (y\bullet x)&=\rho_x(y\bullet x)=(y*x)\bullet x=\rho_{y}^{-1}((y*((x*y)\bullet y))\bullet (x*y))=(y*(x\bullet y))\bullet x.
\end{align*}
and so \eqref{ax42b} is equivalent to \eqref{eq_bondles}.

\end{proof}

The identity \eqref{eq_bondles} does not hold for every oriented singquandles. Let $Q=\{1,2,3\}$ and
$$
(Q,\bullet)=\begin{tabular}{| c c c|}
\hline
2 & 1 &3\\
2 & 2 & 2\\
1& 3 &2\\
\hline
\end{tabular}\,,\quad (Q,*)=(Q,/)=\begin{tabular}{| c c c|}\hline
1 & 3 &1\\
2 & 2 & 2\\
3& 1 &3 \\\hline
\end{tabular}\,,
$$
Then $(Q, \bullet, *,/)$ is a oriented singquandles that does not satisfies \eqref{eq_bondles}.

\begin{example}
Let $(Q,\circ,\bullet,*,/)$ be an oriented bondle. If the right multiplication mappings with respect to $\bullet$ are injective, then $y* x=y*(x \bullet y)$. If $\bullet=*$ then $y*x=y*(x*y)=(y*x)*y$. Then $(Q,*,*,*,/)$ is a bondle if and only if $\rho_y^2=\rho_{y*x} \rho_{y/x}$ and $y*x\approx (y*x)*y$ (or equivalently $y\approx y*(y/x)$) holds. If $(Q,*,/)$ is involutory, i.e. $*=/$ we have that $(Q,*,*,*,*)$ is a bondle if and only if $y\approx y*(y*x)$ holds.
\end{example}

\section{2-reductive quandles}

A quandle is {\it $2$-reductive} if the following identity hold:
$$x*(y*z)\approx x*(y*u).$$
%
%

In \cite{Medial}, a construction for such quandles was provided: let $\setof{A_i}{i\in I}$ a family of abelian groups, $\setof{c_{i,j}}{i,j\in I}$ such that $c_{i,j}\in A_j$ for every $i,j\in I$ and $A_j=\langle c_{i,j},\, i\in I\rangle$. We define two binary operations on the set $\mathcal{A}=\bigcup_{i\in I} A_i$ as
$$x*y=x+c_{i,j},\quad x/y=x-c_{i,j}$$
whenever $x\in A_j$ and $y\in A_i$. Then $(\mathcal{A},*,/)$ is a quandle called {\it affine mesh} and we denote it by $((A_i)_{i\in I},(c_{i,j})_{i,j\in I}))$. Affine meshes completely describe $2$-reductive quandles.

\begin{proposition}\cite[Theorem 6.9]{Medial} \cite[Theorem 4.1]{Dis}
Let $(Q,*,/)$ be a quandle. The following are equivalent:
\begin{itemize}
\item[(i)] $(Q,*,/)$ is $2$-reductive.
\item[(ii)] $\lmlt(Q,*,/)$ is abelian.
\item[(iii)] $(Q,*,/)$ is isomorphic to an affine mesh $((A_i)_{i\in I},(c_{i,j})_{i,j\in I}))$.
\end{itemize}
\end{proposition}

In particular, note that if $\lmlt(Q,*,/)$ is abelian, then $\lmlt(Q,*^n,/^n)$ is also abelian for every $n\in \mathbb{Z}$.

Let us conclude the section showing that $2$-reductive quandles can be used to construct examples of oriented singquandles and oriented bondles.

\begin{lemma}\label{oriented lmlt abelian}
Let $(Q,*,/)$ be a $2$-reductive quandle. Then $(Q,*^n,*^m,/^m)$ is an oriented singquandle for every $n,m\in \mathbb{Z}$.
\end{lemma}

\begin{proof}
Clearly $\rho_x\in \aut{Q,*^n,/^n}$, i.e. \eqref{Oriented1} holds. Moreover $\rho_x=\rho_{h(x)}$ for every $x\in Q$ and every $h\in \lmlt(Q)$, and so we have:
$$(z*^m x)*^m y=(z*^m (x*^n y))*^m ((y*^m x)*^n x),$$
i.e. also \eqref{Oriented2} holds.
\end{proof}

\begin{lemma}
Let $(Q,*,/)$ be a $2$-reductive quandle. Then $(Q,*^n,*^m,*,/)$ is a bondle for every $n,m\in \mathbb{Z}$.
\end{lemma}

\begin{proof}
According to Lemma \ref{oriented lmlt abelian}, $(Q,*^n,*,/)$ and $(Q,*^m,*,/)$ are oriented singquandles. Moreover since $\rho_{x*^m y}=\rho_x$ we have that $y* x=y* (x*^m y)$.
\end{proof}

\bibliographystyle{amsalpha}
\bibliography{references} 

\end{document}